\documentclass[12pt]{article}
\usepackage{amsmath, amsthm, amssymb, amsfonts}
\usepackage{psfrag, epsfig}
\newtheorem{thm}{Theorem}[section]

\newtheorem{lem}[thm]{Lemma}
\newtheorem{conj}[thm]{Conjecture}

\theoremstyle{remark}

\theoremstyle{definition}

\newtheorem{defn}{Definition}[section]

\newcommand{\A}{\mathcal{A}}
\newcommand{\piinv}{\pi^{-1}}
\newcommand{\Pres}{\mathcal{P}}
\newcommand{\M}{\mathcal{M}}

\newcommand{\mubar}{\bar{\mu}}
\newcommand{\XS}{X\times S}
\newcommand{\Part}{\mathcal{P}}
\newcommand{\R}{\mathcal{R}}
\newcommand{\Qtilde}{\tilde{\mathcal{Q}}}
\newcommand{\Q}{\mathcal{Q}}
\newcommand{\F}{\mathcal{F}}
\newcommand{\Calg}{\mathcal{C}}
\newcommand{\Balg}{\mathcal{B}}
\newcommand{\dbar}{\bar{d}}
\newcommand{\muhat}{\hat{\mu}}
\newcommand{\xbar}{\bar{x}}

\newcommand{\Z}{\mathbb{Z}}
\newcommand{\U}{\mathcal{U}}

\title{Compensation functions for factors of shifts of finite type}
\author{John Antonioli}
\date{}
\begin{document}
\maketitle

\begin{abstract}
Let $\pi:X \to Y$ be an infinite-to-one factor map, where $X$ is a shift of finite type.  A compensation function relates equilibrium states in $X$ to equilibrium states on $Y$.  The $p$-Dini condition is given as a way of measuring the smoothness of a continuous function, with $1$-Dini corresponding to functions with summable variation.  Two types of compensation functions are defined in terms of this condition.  Given a fully-supported invariant measure $\nu$ on $Y$, we show that the relative equilibrium states of a $1$-Dini function $f$ over $\nu$ are themselves fully supported, and have positive relative entropy.  We then show that there exists a compensation function which is $p$-Dini for all $p > 1$ which has relative equilibrium states supported on a finite-to-one subfactor.
\end{abstract}

\section{Introduction}

Nearly everywhere symbolic dynamical systems arise, questions about factor maps between subshifts have been studied.  In studying the thermodynamic formalism of these factor systems one would generally like to understand the structure of equilibrium states which are maximal relative to some invariant measure in the factor.  Compensation functions arise in this setting as a useful tool for categorizing these relative equilibrium states.  

For a factor triple $ F = ( X, Y, \pi ) $, a compensation function is a function $f \in C(X) $ which satisfies $\Pres_{X}( f + \phi \circ \pi ) = \Pres_{Y}( \phi )$ for all $\phi \in C(Y)$, where $\Pres$ is the topological pressure.  A compensation function has been described as an ``oracle for how entropy can appear in a fiber'' \cite{BP}.  The meaning of this statement is made clearer by the following fact (first observed in \cite{BT}).  If $\nu \in \M(Y)$ then
\[
\sup_{\mu \circ \piinv = \nu} \left\{ h(\mu) - h(\nu) + \int f d\mu \right\} = 0
\]
When $\pi$ is infinite-to-one we know that $X$ is inherently capable of generating more topological entropy than $Y$, so we might expect $h(\mu) - h(\nu)$ to be positive for measures which maximize this expression.  In this way, it seems the function $f$ is ``compensating'' for the relative entropy over $\nu$.  In fact for a shift of finite type, if we define 
\[
g(y) = \limsup_{n \to \infty} \frac{1}{n} \log \vert \piinv(y_0 ... y_{n-1}) \vert
\]
this function agrees with the usual relative topological entropy a.e. with respect to every invariant measure, and $-g \circ \pi$ serves as a measurable, though likely not continuous, compensation function \cite{YY}.

One general strategy for employing compensation functions is to prove that one with some structure exists for a factor triple, and that this structure enforces certain conditions on the relative equilibrium states which arise.  Saturated compensation functions are those which can be written as $g \circ \pi$ such as the relative topological entropy above.  They have been studied in \cite{SS} and \cite{YY} for their relation to weighed entropy and Hausdorff measure.  In \cite{Wcompfn} it was shown that factors with a compensation function in the Walters class of continuous functions have nice lifting properties for equilibrium states.

In the original survey of compensation functions by Walters \cite{Wcompfn}, it is asserted that every subshift factor triple possesses at least a continuous compensation function.  However, the type of function proved to exist is very different from the ``oracle'' above.  Walters' compensation functions achieve their task by forcing relative equilibrium states to live on a finite-to-one subfactor.  No relative entropy can be generated on this subfactor, so the pressures are equal in a more trivial way.  In general no such subfactor need exist for a subshift factor triple, which has caused this statement to be disputed.  In the case where $X$ is at least sofic, the existence of a finite-to-one subfactor is guaranteed by \cite{MPW}.  These competing notions of compensation lead us to define two types of compensation function.

\begin{defn} \label{def:oraclecomp} Let $(X, Y, \pi)$ be a subshift factor triple.  Let $f$ be a compensation function for this factor triple.  Then $f$ is an \emph{oracle-type} compensation function if for every fully supported $\nu \in \M(Y)$, the relative equilibrium states of $f$ over $\nu$ are fully supported and have positive relative entropy.
\end{defn}

\begin{defn} \label{def:walterscomp} Let $(X, Y, \pi)$ be a subshift factor triple.  Let $f$ be a compensation function for this factor triple. Then $f$ is a \emph{Walters-type} compensation function if there exists a subshift $X_{0} \subset X$ with $\pi \vert_{X_{0}}$  surjective and finite-to-one such that the relative equilibrium states of $f$ are supported on $X_{0}$.  In this case the relative equilibrium states of $f$ will have zero relative entropy.
\end{defn}

A Walters-type compensation function must be sharply negative off the subfactor to scare equilibrium states away from trying to generate relative entropy.  A natural question to ask is how sharp they must be.  Many known oracle-type compensation functions have summable variation.  This leads one to conjecture that functions of summable variation are not sharp enough.

In this work an attempt to encompass weaker continuity conditions leads to inspecting functions for which $( \text{var}_{n}(f) )^p$ is summable, the so called $p$-Dini functions.  For an infinite-to-one factor map of an irreducible shift of finite type, our primary results will show that a $1$-Dini compensation function must be of oracle-type, whereas we can explicitly construct a Walters-type compensation function which is simultaneously $p$-Dini for every $p>1$.  This fact exposes a phase transition in the relative equilibrium states of compensation functions.  These statements are made precise in Theorems \ref{thm:1dini} and \ref{thm:pdini}.

\subsection*{\small {\em Acknowledgement}}
This paper is based on the PhD Thesis of the author.  The author gratefully acknowledges the support and encouragement of his adviser Anthony Quas, as well as Karl Petersen and Chris Bose for their helpful comments and discussions.

\section{Preliminaries}

We will always be working with a factor triple $F = ( X, Y, \pi )$ where $X$ is an irreducible shift of finite type and $\pi$ is a $1$-block code to a sofic shift $Y$.  In the proofs, we will also make the assumption that $X$ is a $1$-step SFT by a standard higher-block re-encoding argument.  The shift dynamics on $X$ and $Y$ will be called $T$ and $S$ respectively.

For a word $w$ of length $n$ in $X$ (or some other shift space) we will let $[w]_{i} = \{x \in X \mid x_{i} \dots x_{i+n-1} = w \}$.  This is just the cylinder set of points having $w$ starting at coordinate $i$.  When the subscript is dropped, we assume $[w] = [w]_{0}$.  Sometimes we will use $[w]_{i}^{j}$ to make explicit both the start and end points of our cylinder set.

Several constructions rely on the existence of a subshift $X_0 \subset X$ such that $\pi \vert_{X_0}$ is finite-to-one.  When $X$ is an irreducible shift of finite type $X_{0}$ is given by a construction which appears in \cite{MPW}.  One first puts an order on the symbols in the alphabet of $X$.  This imparts a lexicographic order on words in $X$ which have the same endpoints and image under $\pi$.  Sets of words which share this property are said to form \emph{a diamond over} $Y$.  Let $\U$ be the collection of all words which are lexicographically minimal in the ordering on diamonds.  A point $x$ is in $X_{0}$ if every subword of $x$ is in $\U$.  This condition prevents two words which form a diamond from appearing in $X_{0}$, which ensures the code is finite-to-one.  We will sometimes refer to such an $X_{0}$ as a \emph{MPW-subshift} of $X$ and words in $\U$ as \emph{MPW-minimal}.

There are several useful notions of relative pressure which stem from their introduction in \cite{LW} and further development in \cite{Wcompfn}.  One which has particular relevance for compensation functions is the {\em maximal relative pressure}.  It is defined for $f \in C(X)$ as $W(f) = \sup_{\mu \in \M(X)}{ \left\{ h(\mu)-h(\mu \circ \piinv)+\int{f \, d\mu} \right\} }$.  It can be shown that this is equivalent to $\sup \left\{ \Pres( f + \phi\circ\pi) - \Pres(\phi) \mid \phi \in C(Y) \right\}$, and so it is clear that if $f$ is a compensation function then $W(f) = 0$.  Equilibrium states in the factor setting will generally be defined relative to some measure in $Y$.  We will say $\mu \in \M(X)$ is a {\em relative equilibrium state of $f$ over $\nu \in \M(Y)$} if it obtains the supremum in $\sup_{\mu} \left\{ h(\mu) - h(\nu) + \int f \, d\mu \mid \mu \circ \piinv = \nu \right\}$.  We will also use the notation $h(\mu) - h(\mu \circ \piinv) = h(\mu \mid \mu \circ \piinv)$ for the relative entropy of a measure.  Good references for the general theory of pressure and equilibrium states are \cite{Wbook} and \cite{K}.

In our setting it makes sense to define $\text{var}_{n}(f)$ to be the maximum distance two points which agree on a $2n$-window about the origin can be mapped apart.  In other words, $\text{var}_{n}(f) = \max_{d(x,y) \leq 2^{-n}} \left\vert f(x) - f(y) \right\vert$.  We will say a function $f$ is \textit{ $p$-Dini} if $\sum_{n} \left( \text{var}_{n}(f) \right)^p < \infty$.  Note that the $1$-Dini functions are exactly those with summable variation.

We note that if we take the trivial factor, Theorem \ref{thm:1dini} gives the well known non-relative statement that 1-Dini functions have fully supported equilibrium states \cite{Wgmeas}.  The nonrelative case also provides inspiration for the use of the $p$-Dini condition. In the nonrelative case, an example of a function with nonunique equilibrium states given by Hofbauer \cite{H} is similar to the one constructed for the second proof in the simple case of $X$ being the full $2$-shift and taking a trivial factor.  This function is not $1$-Dini, but it is $p$-Dini for $p > 1$.  Additionally, in \cite{GS} a similar phase transition is seen as the modulus of continuity of a piecewise expanding map of the interval is relaxed from one which is summable to one which is summable when raised to a power $p$ which is strictly greater than $1$.

Several lemmas from \cite{Yoo} will also be needed in the proof of the first theorem.  The first lemma has to do with the ability to construct diamonds which live half-in and half-out of a proper subshift.
\begin{lem} \label{yoopaths}
Let $\pi:X \to Y$ be an infinite-to-one, 1-block code from an irreducible shift of finite type $X$.  If $Z \subset X$ is a proper subshift such that $\pi(Z)=Y$, then there exist two words $u$ and $v$ in $X$ with the following properties:
\begin{enumerate}
\item $\pi(u)=\pi(v)$
\item $\vert u \vert = \vert v \vert = n$, $u_0=v_0$ and $u_{n-1}=v_{n-1}$
\item $u$ is a word appearing in $Z$, $v$ does not occur in $Z$
\item For any words $s$ and $t$ with $sut$ a word in $Z$, there is exactly one occurrence of $v$ in $svt$.
\end{enumerate}
\end{lem}
The second two lemmas actually apply to any invariant measure $\mu$ on a $\sigma$-algebra $\F$.  For a set $A \in \F$ (or a sub-$\sigma$-algebra of $\F$) with $\mu(A) > 0$ we let $\F_{A} = \{ B \cap A \mid B \in \F \}$, and use similar notation for the restriction of sub-$\sigma$-algebras and partitions to $A$.  Similarly, we let $\mu_{A}(B) = \frac{\mu(B)}{\mu(A)}$ and for a partition $\Part$ and sub-$\sigma$-algebra $\Balg$ we let $H_{A}(\Part \mid \Calg) = H_{\mu_{A}}(\Part_{A} \mid \Balg_{A})$.
\begin{lem}\label{yooindlemma}
Let $\Part$ be a measurable partition of $X$ and let $\Balg$ and $\Calg$ be sub-$\sigma$-algebras of $\F$ such that $\Calg$ and $\Part \vee \Balg$ are independent.  If $C \in \Calg$ and $\mu(C)>0$, then
\[
H_{C}(\Part \mid \Balg\vee \Calg) = H(\Part \mid \Balg)
\]
\end{lem}

\begin{lem}\label{yooreslemma}
Let $A \in \F$ and $\mu(A)>0$.  Then for a measurable partition $\Part$,
\[
H(\Part \mid \Balg) \geq \mu(A)H_{A}(\Part \mid \Balg)
\]
\end{lem}

A type of interleaving argument similar to the ones used in \cite{PQS} and \cite{Yoo} is employed to prove Theorem \ref{thm:1dini}.  Some of the estimates used will rely on the continuity of entropy with respect to the $\dbar$ distance.  This result can be found in \cite[Theorem 7.9]{Rud}.
\begin{thm} \label{dbarent}
If $\dbar(\mu, \nu) < \epsilon$ then
\[
\vert h(\mu) - h(\nu) \vert = O(\epsilon \log \epsilon)
\]
\end{thm} 

\section{Oracle-type compensation functions}

Our first theorem deals with the relationship between oracle-type compensation functions and the class of $1$-Dini functions.  Actually, the conclusions of this theorem apply to all $1$-Dini functions, and therefore extend a previously known result for nonrelative equilibrium states.

\begin{thm} \label{thm:1dini}
Let $\pi:X \to Y$ be an infinite-to-one, 1-block code from an irreducible shift of finite type $X$ to sofic shift $Y$.  Let $f \in C(X)$ be a $1$-Dini function and $\nu \in \M(Y)$ be fully supported.  Suppose $\mu \in \M(X)$ is a relative equilibrium state of $f$ over $\nu$.  Then $\mu$ is fully supported and has positive relative entropy over $\nu$.  In particular, if $f$ is a compensation function then it is oracle-type.
\end{thm}

The proof of the first theorem will consist of two main pieces.  Our $1$-Dini function will be $f$.  First we will show that every relative equilibrium state of $f$ over a fully supported $\nu$ is itself fully supported.  Then we show that if $\mu$ is a relative equilibrium state of $f$ which is fully supported, $ h(\mu)-h(\mu\circ\piinv) > 0$.  

The first part is proved by assuming the proposed relative equilibrium state $\mu$  is not fully supported.  We are then be able to spread measure around in a way which increases relative entropy.  This will also change the integral of $f$.  However, we will exploit $f$ being 1-Dini to bound the change in the integral and show that it is smaller than our entropy gain.  Lemma \ref{yoopaths} allows us to spread our old measure around in a way which is invariant over $Y$, and which is detectable by our new measure.  The second half of the proof is accomplished in much the same manner.
 
\begin{proof}[Proof of the first part of Theorem \ref{thm:1dini}]

Let $\nu$ be a fully supported measure on $Y$ and $\mu$ a relative equilibrium state of $f$ over $\nu$.  Thus $\mu$ is one of the measures for which the relative pressure of $f$ is equal to $ h(\mu)-h(\nu) + \int f \, d\mu$.  Assume that $\mu$ is not fully supported.  Then by applying Lemma \ref{yoopaths} to $Z = \text{spt}(\mu)$, we obtain two words $u$ and $v$ which have the same image and meet at their endpoints.  In addition, $u$ occurs in Z and $v$ does not, which implies $\mu(u)>0$ and $\mu(v)=0$.  So blocks of $u$ occur with positive frequency in typical points for $\mu$.  

The method of constructing the new higher entropy measure $\mubar$ is to occasionally swap $u$ for $v$, spreading measure outside of $Z$ and generating entropy.    

In order to avoid ambiguity about the order $u$'s get swapped, because they may overlap, we make sure swapping does not occur closer than $|u|=n$ steps.  The fourth property of the lemma makes sure that when we see a $v$ in the image of a point from $Z$, we know that if we swap it back to a $u$ we recover the original point.

To be more precise, let $0<p<1$ and consider the full 2-shift $\Omega = \{0,1\}^{\mathbb{Z}}$.   Let $\psi:\Omega \to S$ be a factor map defined by letting $\psi(\omega)_0 = 1$ when $\omega_{-n+1}^{0} \in [0^{n}1]_{-n+1}^{0}$.  Otherwise, $\psi(\omega)_0 = 0$.  Now we construct a measure $\eta$ on $S$ by pushing forward a Bernoulli measure that has $1$ occurring with frequency $p$.

In order to make a map for switching $u$'s into $v$'s, we construct $\phi:X\times S \to X$ as follows.  If $s_i=1$ and $T^{i}(x) \in [u]$ then $\phi(x,s)_{i}^{i+n-1}=v$.  Otherwise $\phi(x,s)_i = x_i$.  The construction of $S$ ensures that all $1$'s are at least $n$ apart.  As discussed above, this implies that we can recover the original point and thus $\phi$ is well defined and commutes with the natural shift on $X \times S$.  Furthermore, we have $\pi(x)=\pi(\phi(x,s))$  for all $(x,s) \in X\times S$.   In addition, if $x \in Z$ then $\phi(x,s)_{-n+1}^{n-1}$ determines $x_0$.  This is due to the last two properties of Lemma \ref{yoopaths}.

The measure $\mubar$ will be a push forward of $\mu \times \eta$ on $X \times S$.   We need to compare the relative entropies of $\mu$ and $\mubar$, which is the same as comparing their entropies because both measures live over $\nu$.  Let $\Q$ be the partition of $X\times S$ generated by the first symbol of $\phi(x, s)$.  Let $\Part$ and $\R$ be the state partitions of $X$ and $S$.  Where appropriate, these will also stand for the pullback of the partitions to $\XS$.  Using the notation $\Part_i^j = \bigvee_{k=i}^{j} T^{k}(\Part)$, we can express the three principal entropies as:
\begin{align*}
h(\mubar) &= H(\Q \mid \Q_{-\infty}^{-1}) \\
h(\mu) &=  H(\Part \mid \Part_{-\infty}^{-1}) \\
h(\eta) &= H(\R \mid \R_{-\infty}^{-1}) \\
\end{align*}

Note that for $\mu\times \eta$-a.e. point in $\XS$, $x_0$ is determined by $\phi(x,s)_{-n+1}^{n-1}$.  So the measurable partition $\Part$ generated by $x_0$ is coarser than $\Qtilde = \Q_{-n+1}^{n-1}$.  This allows us to apply Pinsker's formula to say that

\begin{align*}
h(\mubar) &= H(\Qtilde \mid \Qtilde_{-\infty}^{-1})\\
&= H(\Part \mid \Part_{-\infty}^{-1}) + H(\Qtilde \mid \Part_{-\infty}^{\infty}\vee \Qtilde_{-\infty}^{-1})\\
&= h(\mu) + H(\Q \mid \Part_{-\infty}^{\infty}\vee \Q_{-\infty}^{-1})
\end{align*}

We note that 
\begin{align*}
H(\Q_{0}^{n-1} \mid \Part_{-\infty}^{\infty}\vee \Q_{-\infty}^{-1}) = & H(\Q_{0} \mid \Part_{-\infty}^{\infty}\vee \Q_{-\infty}^{-1}) +  H(\Q_{1} \mid \Part_{-\infty}^{\infty}\vee \Q_{-\infty}^{0}) \\
&+ \hdots +  H(\Q_{n-1} \mid \Part_{-\infty}^{\infty}\vee \Q_{-\infty}^{n-2}) =n H(\Q \mid \Part_{-\infty}^{\infty}\vee \Q_{-\infty}^{-1})
\end{align*} 

 Now applying Lemma \ref{yooreslemma} allows us to say that

\begin{align*}
 H(\Q &\mid \Part_{-\infty}^{\infty}\vee \Q_{-\infty}^{-1}) = \frac{1}{n} H(\Q_{0}^{n-1} \mid \Part_{-\infty}^{\infty}\vee \Q_{-\infty}^{-1}) \\
&\geq \frac{\mu([u])}{n} H_{[u]\times S}(\Q_{0}^{n-1} \mid \Part_{-\infty}^{\infty}\vee \Q_{-\infty}^{-1})
\end{align*}

When we restrict our view to $[u]$, $\Q_{0}^{n-1}$ is either $u$ or $v$, and this entirely depends on $s_0$.  In addition, for $x \in [u]$, $\Part_{-\infty}^{\infty}\vee \R_{-\infty}^{-1}$ determines $\Q_{-\infty}^{-1}$.  Thus
\begin{align*}
 H_{[u]\times S}(\Q_{0}^{n-1} &\mid \Part_{-\infty}^{\infty}\vee \Q_{-\infty}^{-1}) \\
&=  H_{[u]\times S}(\R \mid \Part_{-\infty}^{\infty}\vee \Q_{-\infty}^{-1}) \\
&\geq H_{[u]\times S}(\R \mid \Part_{-\infty}^{\infty}\vee \R_{-\infty}^{-1}) \\
&= H(\R \mid \R_{-\infty}^{-1}) \\
&= h(\eta)
\end{align*}
The equality in the fourth line is due to Lemma \ref{yooindlemma}

In the end, we have shown that
\begin{equation} \label{1stest}
H(\Q \mid \Part_{-\infty}^{\infty}\vee \Q_{-\infty}^{-1}) \geq \frac{\mu([u])}{n}h(\eta)
\end{equation}

Now we bound the entropy of $\eta$ by comparing it to the Bernoulli measure $b$ on $\Omega$.  We will compare the measures using the $\bar{d}$ metric.  Let $J(\eta,b)$ be the set of all joinings of $\eta$ and $b$.  Let $\delta_{0}(s,\omega)$ be $1$ if $s_0 \neq \omega_0$ and $0$ otherwise.  Then the $\dbar$ distance between $\eta$ and $b$ can be written as
\[
\dbar(\eta,b) = \inf_{\hat{\mu} \in J(\eta,b)} \int{ \delta_{0}(s,\omega) \, d\hat{\mu}}
\]
Let us construct a joining of $\eta$ and $b$.  For cylinder sets $C \subset S$ and $D \subset \Omega$ let $\hat{\mu}(C\times D) = b(\psi^{-1}(C) \cap D)$.  This measure is invariant under $\sigma \times \sigma$.  We also have that $\hat{\mu}(C \times \Omega) = b(\psi^{-1}(C) \cap \Omega) = \eta(C)$ and $\muhat(S \times D) = b(D)$.  So $\muhat$ is a joining, and thus
\[
\dbar(\eta,b) \geq \int \delta_{0}(s,\omega) \, d\hat{\mu}
\]
If $(s,\omega)$ is a typical point for $\muhat$ then $s_0 = \omega_0$ except when $\omega_0 = 1$ and $\omega_{-n+1}^{-1}$ contains at least one additional $1$, which is an $O(p^2)$ event.  We know $1$'s occur with probability $p$ in $b$, thus
\[
\dbar(\eta,b) = O(p^2)
\]
As we will be constructing a series of upper and lower bounds, it will be convenient to introduce the terminology $f(n)$ is  $\Omega^{+}(g(n))$ if $f(n) \geq \vert c g(n) \vert$ for some constant $c$. From Theorem \ref{dbarent} we have
\[
h(\eta)-h(b) = O(p^2 \log p)
\]

Combining this with our estimate \eqref{1stest}, we have
\[
h(\mubar) - h(\mu) = \frac{\mu([u])}{n}(\Omega^{+}(p \log p)+ O(p^2 \log p)) = \Omega^{+}(p \log p)
\]

To bound the difference of the integrals of $f$, note that for $\mu \times \eta$-a.e. $(x,s) \in \XS$, $x$ is a typical point for $\mu$ and $\phi(x,s) = \bar{x}$ is a typical point for $\mubar$.  So the ergodic theorem tells us that
\[
   \left\vert \int f \, d\mu - \int f \, d\mubar \right\vert \leq \lim_{N \to \infty} \frac{1}{N} \sum_{i=0}^{N-1}\left\vert f(T^{i}x) - f(T^{i}\xbar) \right\vert
\]
The only places where $x$ and $\xbar$ differ are where a $u$ from $x$ has been changed to a $v$ in $\xbar$.  This happens on average $\mu([u])\eta([1])N$ times in a run of $N$.  Since $f$ is $1$-Dini we can say $\sum_{n} \left( \text{var}_{n}(f) \right) = L$.  So the most $f(T^{i}x)$ and $f(T^{i}\xbar)$ can differ by is $L$ and each $v$ can contribute at most $nL$ to the sum above.  Accounting for possible tail contributions, we have that for all $\epsilon > 0$ there is an $N_{0}$ such that for $N > N_{0}$
\[
 \frac{1}{N}\sum_{i=0}^{N-1}\left\vert f(T^{i}x) - f(T^{i}\xbar) \right\vert \leq \mu([u])\eta([1])Ln + \frac{2L}{N} + \epsilon
\]
Taking limits and using the fact that $\eta([1]) = O(p)$, we estimate the integral as
\[
 \left\vert \int f \, d\mu - \int f \, d\mubar \right\vert = O(p)
\]

Finally, since the entropy term is $\Omega^{+}(p \log p)$ and the integral is only $O(p)$ we have that for $p$ small enough, $h(\mubar) + \int f \, d\mubar - h(\mu) - \int f \, d\mu > 0$.  However, $\mu$ was assumed to have maximal relative entropy over $\nu$.  This proves that $\mu$ is fully supported.
\end{proof}

Now we must prove that the now-fully supported relative equilibrium state $\mu$ has positive relative entropy.

\begin{proof}[Proof of the second part of Theorem \ref{thm:1dini}]
Assume that $h(\mu \mid \nu) = 0$.  Proceeding similarly to the first part of the proof we will construct a new measure $\mubar$ which has greater relative pressure.  Our estimate will be showing that the difference in relative entropies is greater than in the integrals. 

Once again, $\mubar$ is constructed by taking a pair of words $u$ and $v$ which can be swapped in a point without affecting the eventual image in $Y$.  Here, we do not need the strong detectability conditions provided by Lemma \ref{yoopaths}.  We can find such a pair of paths from the fact that $\pi$ is infinite-to-one, and thus there exist a pair of paths $u$ and $v$ which form a diamond over $Y$.    

 Let $\vert u \vert = n$.  Construct the approximately Bernoulli process $(S,\sigma,\eta)$ as before.  Thus no two $1$'s occur in $S$ closer than $n$.   We construct $\mubar$ by pushing $\mu \times \eta$ forward by $\phi$ which changes $u$ into $v$ and $v$ into $u$ whenever $s_0 = 1$.

The entropy term we need to estimate is $h(\mubar \mid \nu) - h(\mu \mid \nu) = h(\mubar \mid \nu)$.  Let $\mathcal{P}$ be the partition on $\XS$ from the $(X,\mu)$ process, $\bar{\mathcal{P}}$ from the $(X,\mubar)$ process, $\mathcal{Q}$ from the $(S,\eta)$ process and $\mathcal{R}$ from the $(Y,\nu)$ process.  Then

\begin{align*}
h(\mubar \mid \nu) &= H(\bar{\mathcal{P}}_0 \mid \bar{\mathcal{P}}_{-\infty}^{-1} \vee \mathcal{R}_{-\infty}^{\infty}) \\
&\geq H(\bar{\mathcal{P}}_0 \mid \bar{\mathcal{P}}_{-\infty}^{-1} \vee \mathcal{P}_{-\infty}^{\infty}) \\
&\geq \mu([u]\cup[v])H_{[u]\cup[v] \times S}(\bar{\mathcal{P}}_0 \mid \bar{\mathcal{P}}_{-\infty}^{-1} \vee \mathcal{P}_{-\infty}^{\infty}) \\
&=\frac{ \mu([u]\cup[v])}{n}H_{[u]\cup[v] \times S}(\bar{\mathcal{P}}_{0}^{n-1} \mid \bar{\mathcal{P}}_{-\infty}^{-1} \vee \mathcal{P}_{-\infty}^{\infty})
\end{align*}

The third inequality is from Lemma \ref{yooreslemma}.  In this last term, we have restricted our view to being at either $u$ or $v$.  Knowing $\mathcal{P}_{-\infty}^{\infty}$ means that finding out which one $\bar{\mathcal{P}}_{0}^{n-1}$ is exactly tells us if $s_0 = 1$ or $s_0 = 0$.  This allows us to say
\begin{align*}
H_{[u]\cup[v] \times S}(\bar{\mathcal{P}}_{0}^{n-1} \mid \bar{\mathcal{P}}_{-\infty}^{-1} \vee \mathcal{P}_{-\infty}^{\infty}) &= H_{[u]\cup[v] \times S}(\mathcal{Q}_0 \mid \bar{\mathcal{P}}_{-\infty}^{-1} \vee \mathcal{P}_{-\infty}^{\infty}) \\
&\geq  H_{[u]\cup[v] \times S}(\mathcal{Q}_0 \mid \mathcal{Q}_{-\infty}^{-1} \vee \mathcal{P}_{-\infty}^{\infty}) \\
&= H(\mathcal{Q}_0 \mid \mathcal{Q}_{-\infty}^{-1}) \\
&=h(\eta)
\end{align*}
Here, the last equality is due to the independence of $X$ and $S$, and a final use of Lemma \ref{yooindlemma}.  We have already calculated $h(\eta)$, and so our estimate on the relative entropies is
\[
h(\mubar \mid \nu) - h(\mu \mid \nu) = \Omega^{+}(p \log p)
\]

From here, we must again estimate the integral term.  One way to do this is to construct a joining of $\mu$ and $\mubar$.  Our joining will be $\muhat$ defined on products of cylinder sets from $X$  as $\muhat(A \times B) = \mu(\pi_{X}^{-1}(A) \cap \phi^{-1}(B))$.  Note that one way to construct a typical pair of points for $\muhat$ is to pick $(x, \xbar)$ from $X \times X$ where $x$ is typical for $\mu$ and $\xbar = \phi(x,s)$ for some typical $s$.  The number of times $x$ and $\xbar$ will then differ is $O(p)$.  

The difference in the integrals is bounded by $\lim_{N \to \infty} A_{N}(\vert f(x)-f(\xbar)\vert)$, where $A_{N}$ is the ergodic average.  In each place where $x$ and $\xbar$ differ we receive a constant bounded contribution to the sum because $f$ is 1-Dini.  Each tail end also gives a bounded contribution, which is constant in $N$.  The contribution from differences is $O(Np)$ and the tail contribution is $O(1)$ so in the limit we have that the integral is $O(p)$.  Together with our earlier estimate that the relative entropy gain is $\Omega^{+}(p \log p)$, this tells us that $\mubar$ has strictly greater relative pressure than $\mu$, a contradiction.  
\end{proof}

\section{Walters-type compensation functions}

For our second theorem we will construct a function $f$ which is simultaneously $p$-Dini for all $p>1$, and whose relative equilibrium states live on a MPW-subshift of $X$.  From this it is clear that their relative entropy over measures on $Y$ is zero.

\begin{thm} \label{thm:pdini}
Let $  F = ( X, Y, \pi ) $ be a factor triple with $X$ an irreducible shift of finite type, and $Y$ a factor of $X$ by $1$-block code $\pi$.  Then there exists a Walters-type compensation function $f \in C(X)$ which is $p$-Dini for all $p>1$
\end{thm}

Note that with a different choice of $f$, an argument similar to the one which follows could be used to prove a theorem of the form: \dots for all $p>1$ there exists a Walters-type compensation function $f$ which is $q$-Dini for all $q>p$, but which is not $q$-Dini for $q \leq p$ \dots

It is important to note that, while Walters-type compensation functions are shown to exist under these conditions in \cite{Wcompfn}, the proof of their existence is not constructive.  In particular, it makes no claims about the type of continuity exhibited by the compensation function.  

The main idea of the proof is to construct a finite extension of $X$ based on subdividing points into minimal subwords and considering the space of points paired with the dividers which delineate these subwords.  This extension encodes a sufficient amount of information to reproduce the relative entropy, and is easily partitioned into sets where the integral can be bounded.  By inducing on returns to the dividers we will be able to show that the potential to increase relative entropy by straying from the MPW-subshift is outweighed by the penalty incurred in the integral.

\subsection{Clothespinning sequences}

For points $x$ which are not in the MPW-subshift we will define a set of  ``clothespins'' of $x$ such that the words between adjacent clothespins are minimal, but the words between nonadjacent clothespins are not.  By showing that there are only finitely many distinct ways to accomplish this clothespinning we will be able to define a space of clothespinned points which lives over $X$ in a finite-to-one way.

Throughout this section we will be working with a factor triple $(X, Y, \pi)$ where $X$ is a shift of finite type and $Y$ is a sofic image of $X$ under a $1$-block code $\pi$.  We will assume that some MPW-order has been put on the symbols in $X$ and that $X_0$ is the associated MPW-subshift.  Recall that the set of words which are MPW-minimal (given their endpoints and image) is $\U$.  As in the first proof, let $\Omega = \{0, 1\}^{\Z}$.

\begin{defn} \label{clothespinsdef}
For $x \in X$ consider the following process.  Let $n_{0}^{N}(x) = -N$.  Given $n_{k-1}^{N}(x)$ define $n_{k}^{N}(x) = \min \left\{ i > n_{k-1}^{N}(x) \mid x_{n_{k-1}^{N}} \hdots x_{i+1} \not\in \U \right\} $.  If some $n_{k}^{N}(x) = \infty$, terminate the process.  Define $s^{(N)}(x) \in \Omega$ by $s^{(N)}_{i}(x) = 1$ if $i \in \left\{ n_{k}^{N}(x) \right\}_{k=0}^{\infty}$.  If $s$ is a limit of $s^{(N)}(x)$ then we will call $s$ a \emph{clothespinning sequence} of $x$.  A position $i$ for which $s_i = 1$ will be called a \emph{clothespin} of $x$.

\end{defn}

Note that if $x \in X_{0}$ then $s^{(N)}(x)$ has exactly one clothespin at $-N$ and all $0$'s otherwise.  So the only clothespinning sequence of a point in $X_{0}$ is the fixed $0$ sequence.

\begin{figure}[h]
\centering{
\psfrag{$x$}{$x$}
\psfrag{$s$}{$s$}
\psfrag{...0010...}{\dots 0010 \dots}
\psfrag{$n_i$}{$n_{i}$}
\psfrag{$n_i+2$}{$n_{i+2}$}
\psfrag{$n_i+1$}{$n_{i+1}$}
\psfrag{$w$}{$w$}
\psfrag{$pi(x)$}{$\pi(x)$}
\includegraphics[height=1.5in]{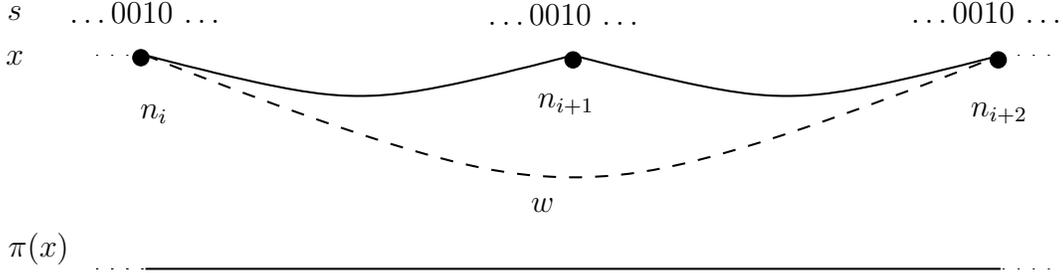}}
\caption{Between adjacent clothespins $x$ is minimal, but if we pull out $n_{i+1}$ we see the word $w$ hangs lower than $x_{n_{i}}^{n_{i+2}}$}
\label{fig:clothespins}
\end{figure}

The most important property of clothespinning sequences is that if $s_{i}^{j} = 10^{j-i-1}1$ then $x_{i}^{j} \in \U$, but $x_{i}^{j+1} \not\in \U$. A few additional facts about clothespinnings are summarized in the following lemma.
\begin{lem} \label{finiteclothespins}
Let $\mu \in \M(X)$ be an ergodic measure such that $\mu(X \setminus X_0 ) > 0$.  Then there exists a $k$ and a set $CP(X)$ with $\mu( CP(X) ) = 1$ such that every point in $CP(X)$ has exactly $k$ distinct clothespinning sequences.
\end{lem}

\begin{proof}
Let $CP_{0}(X) = \left\{ x \in X \mid \forall \, i \in \Z \, \exists \, j>i \, \text{ such that } \, x_{i}^{j} \not\in \U \right\}$.  The ergodic theorem tell us that since $\mu(X \setminus X_0) > 0$, $\mu$-a.e. point must encounter words which are not MPW-minimal infinitely often.  Thus $\mu$-a.e. point is in $CP_{0}(X)$.

Let $x \in CP_{0}(X)$.  Each new clothespin's construction only depends on the previous clothespin, so if $s$ and $s'$ are two distinct clothspinning sequences for $x$ and $s_{i} = s'_{i} = 1$ then $s_{j} = s'_{j}$ for all $j > i$.   Additionally, for any two clothespinnings $s$ and $s'$,  if $i_1$ and $i_2$ are adjacent clothespins in $s$, and $j_1$ and $j_2$ are adjacent in $s'$ with $j_1 \leq i_1$, then $j_2 \leq i_2$. For if $j_2 > i_2$ then $x_{j_1}^{j_2}$ contains $x_{i_1}^{i_{2}+1}$ which is nonminimal.  Thus $s'$ has a pin between $i_1$ and $i_2$, and the position of this pin determines all pins to the right.  From this we can conclude that for any $n \in \Z$ the number of distinct clothespinnings over $x_{n}^{\infty}$ is finite, and decreasing as $n \to \infty$.

Define $k_{n}(x)$ to be $\left\vert \{s_{n}^{\infty} \mid s \text{ is a clothespinning sequence of } x \} \right\vert$.   The preceding explanation shows us that $k_{n}(x)$ is decreasing and $k_{n}(x) \geq 1$, so we can let $k(x) = \lim_{n \to \infty} k_{n}(x)$.  Consider the set $A_{m} = \{ x \in CP_{0}(X) \mid k_{m}(x) = k(x) \text{ and } k_{m-1}(x) > k(x) \}$.  As we observe the clothespinning sequences of $x$ going to toward $\infty$, we lose a sequence whenever two clothespins from previously distinct pinnings coincide.  The set $A_{m}$ represents the set of $x$ whose ``last coincident pin'' occurs at time $m$.  Note that if $x \not\in A_{m}$ for any $m \in \Z$ then $x$ has a finite number of clothespinning sequences in total.  However, $A_{m}$ is not a set we can return to once we shift a point by $T$.  So by the Poincar\'{e} recurrence theorem, $\mu(A_{m}) = 0$ for all $m \in \Z$.  So $\mu$-a.e. point in $CP_{0}(X)$ has finitely many distinct clothespinning sequences.  Now, because $\mu$ is ergodic and $k(x)$ is shift invariant, $k(x)$ is equal to a constant almost everywhere.  In particular, $\mu$-a.e. $x$ has $k_{n}(x) =k$ so in a sense the number of clothespinnings we see locally is constant.  Thus we can define $CP(X) = \{ x \in CP_{0}(X) \mid k_{n}(x) = k \, \forall \, n \}$. 
\end{proof}

The next lemma is needed to ensure that we can build a measure on the clothespinned space which projects to a given measure on $X$.

\begin{lem} \label{clothespinmeasure}  Let $\mu \in \M(X)$ be a measure such that $\mu(X \setminus X_0) > 0$.  Let $\bar{X}$ be the set of pairs $(x,s)$ where $x \in CP(X)$ and $s$ is a clothespinning sequence for $x$ and $\bar{T}(x,s)$ be the usual shift on $x$ and $s$.  Let $\pi_{X}(x,s) = x$.  Then there exists a measure $\bar{\mu} \in \M(\bar{X})$ such that $ \bar{\mu} \circ \pi_{X}^{-1} = \mu$.
\end{lem}

\begin{proof}
First let us note that by Lemma \ref{finiteclothespins} we know the factor map $\pi_{X}$ is finite-to-one.  Let $\mu$ be an ergodic measure such that $\mu(X \setminus X_{0}) > 0$.  By Lemma \ref{finiteclothespins} we know that there is a $k$ such that $\mu$-a.e. point in $CP(X)$ has exactly $k$ clothespinning sequences.

Let $B_{n}(\Omega)$ be the words of length $n$ in $\Omega$.  For every word $v \in B_{n}(\Omega)$ and $x \in CP(X)$ let us define $f_{v}(x) = \left\vert \{ (x,s) \in \bar{X} \mid s_{0}^{n-1} = v \} \right\vert$.  In other words, $f_{v}(x)$ is the number of clothespinning sequences over $x$ that ``look like'' $v$ at the origin.  If we let $w$ be a word of length $n$ in $X$ then for any $x \in [w]$ we have $\sum_{v \in B_{n}(\Omega)} f_{v}(x) = k$.  This is because each of the $k$ clothespinning sequences of $x$ appear in at least one of the $v$'s, and they cannot belong to more than one.  Furthermore, we can see that for any $v \in B_{n}(\Omega)$, $f_{0v}(x) + f_{1v}(x) = f_{v0}(x) + f_{v1}(x) = f_{v}(x)$.

Let $C = \left( [w] \times [v] \right)_{i}^{i+n-1}$ be a cylinder set of length $n$ in $\bar{X}$.  We will define $\bar{\mu}$ on such a cylinder by $\bar{\mu}(C) = \frac{1}{k} \int_{[w]_{0}^{n-1}} f_{v}(x) \, d\mu$.  It is not hard to check the usual consistency conditions to see that $\bar{\mu}$ extends to an invariant measure on $\bar{X}$.  Noting that $\pi_{X}^{-1}([w]) = \bigcup_{v \in B_{n}(\Omega) \cap \bar{X}} [w] \times [v]$ we can then show that
\begin{align*}
\bar{\mu} \circ \pi_{X}^{-1}([w]) &= \frac{1}{k} \int_{[w]} \left( \sum_{v \in B_{n}(\Omega)} f_{v}(x) \right) \, d\mu \\
  &= \frac{1}{k} \int_{[w]} k \, d\mu \\
  &= \mu([w])
\end{align*}
So $\bar{\mu}$ projects to $\mu$.
\end{proof}

\subsection{Proof of Theorem \ref{thm:pdini}}

Let $X_0$ be a MPW subshift of $X$.  This subshift is defined by some ordering on the preimages of symbols in $Y$, which establishes a lexicographic order on diamonds, in $X$.  Let $\U$ be the set of words in $X$ which are minimal in the MPW ordering.  

In order to build a function whose relative equilibrium states live on $X_0$ we need a function which attains its maximum on $X_0$ and which penalizes us heavily for straying from it.  However, the function cannot be too sharp, or it will fail to be $p$-Dini for some $p$ near $1$.  Additionally, we would like for the function to rely on some condition which can be checked on a nice partition around the origin.

For $x \in X$, define $n(x)$ to be the least integer such that $x_{-n(x)} \hdots x_{n(x)} \in \U$ but $x_{-n(x)-1} \hdots x_{n(x)+1} \not\in \U$.  If $x \in X_0$ then define $n(x) = \infty$.  This number gives us a substitute for having the distance from $X_0$ in our function.  

If we want to phrase our function as $f(x) = g(n(x))$ and satisfy that $f$ increases to $0$ as $n \to \infty$ then the $p$-Dini condition becomes $\sum_n g(n)^p < \infty$.  Our choice of function which satisfies that condition will be $f(x) = \frac{-t\log n(x)}{n(x)}$, with $t$ to be determined later.  This function also penalizes us heavily enough to beat the relative entropy which a relative equilibrium state could gain by straying from $X_0$.

Let $\mu$ be a measure on $X$ (a candidate relative equilibrium state) which is not supported on $X_0$.  We will show $h(\mu \mid \mu \circ \piinv) + \int f \, d\mu < 0$ to prove the relative equilibrium states live on $X_0$.  The reason this is sufficient is that for any measure $\mu'$ which is supported on $X_{0}$, the fact that $X_0$ is finite-to-one over $Y$ gives $h(\mu') = h(\mu' \circ \piinv)$.  Since $f \vert_{X_{0}} = 0$, we have $h(\mu') - h(\mu' \circ \piinv) + \int f \, d\mu' = 0$.  This line of justification is similar to the one used by Walters to build continuous compensation functions in \ref{Wcompfn}[Lemma 3.2].

First, we move to the clothespinned space $\bar{X}$ and let $\bar{\mu}$ be the measure defined by Lemma \ref{clothespinmeasure}.  Our next step is to induce on the set $C = \{(x,s) \in \bar{X} \mid s_0 = 1\}$, in other words the clothespinnings with a pin at $0$.  It is clear that $\bar{\mu}(C) > 0$.   We will call the induced system $(X_{C}, T_{C}, \Balg_{C}, \mu_{C})$.  Note that the action of $T_{C}$ is to shift $(x,s) \in C$ so that the next clothespin in $s$ is at the origin. 

If we define $ \bar{f} = f \circ \pi_{X}$ then $\int f \, d\mu = \int \bar{f} \, d\bar{\mu}$.  Let $G(n_{1}, n_{2}, a, b, c)$ be the set of $(x,s) \in C$ such that $n_{1}$ and $n_{2}$ are the first two pins to the right of $0$ with $x_0 = a$, $x_{n_1} = b$ and $x_{n_2} = c$.   Let $G(n_1, n_2) = \bigcup_{a, b, c} G(n_1, n_2, a, b, c)$ and $p(n_{1}, n_{2})=\mu_{C}(G(n_1 , n_2 ))$.   As these partitions indicate, we will be bounding the sum of $\bar{f}$ up to the second return time and to this end we define $f^{2}_{C}(x,s) = \sum_{i=0}^{n_2 - 1} \bar{f}(\bar{T}^{i}(x,s))$.  

If $(x,s) \in G(n_1 , n_2 )$ then we know $x_{-n_2} \hdots x_{n_2} \not\in \U$ because the word between $0$ and $n_2$ is nonminimal.  This allows us to say that $n(x) \leq n_2$, and thus $\bar{f}(x,s) \leq \frac{-t\log(n_2 )}{n_2}$.   In fact, for all $0 \leq i \leq n_2$ we have $n(T^{i}(x)) \leq n_2$ so we bound  $f^{2}_{C}$ as
\begin{align*}
f^{2}_{C}(x,s) &\leq -t \sum_{i=0}^{n_2} \frac{\log{n_2}}{n_2} \\
  &= -t \log{n_2}
\end{align*}
So the bound obtained for the integral is
\[
\int f^{2}_{C} \, d\mu_{C} \leq \sum_{n_1, n_2} -p(n_{1}, n_{2}) B(n_{1}, n_{2}) \\
\]
where $B(n_{1}, n_{2}) = -t \log{n_2}$.

To estimate the relative entropy term, let us first recall that because $\pi_{X}$ is finite-to-one, $h(\mu \mid \mu \circ \piinv) = h(\bar{\mu} \mid \bar{\mu} \circ (\pi \circ \pi_{X})^{-1})$.  So it makes sense to work on bounding the relative entropy in the clothespinned space.  Let $\Q$ be $\sigma$-algebra on $\bar{X}$ given by knowing $\pi \circ \pi_{X}(x,s)_{-\infty}^{\infty}$.  Let $\Q_{C}$ be the elements of $\Q$ intersected with $C$.  Then Abramov's formula tells us that
\begin{align*}
h(\bar{\mu} \mid \bar{\mu} \circ (\pi \circ \pi_{X})^{-1}) &= h( \bar{\mu} \mid  \Q ) \\
&= \frac{1}{2} \bar{\mu}(C) h_{T_{C}^{2}}( \mu_{C} \mid \Q_{C})
\end{align*}
If we know $\pi \circ \pi_{X}(x,s)$ then learning $(x,s)_{0}^{n_{2}}$ is the same as learning $n_{1}$, $n_{2}$,  $x_{0}$, $x_{n_1}$ and $x_{n_2}$.  So if we let $\Part$ be the $G(n_1, n_2, a, b, c)$ partition we can say
\[
 h_{T_{C}^{2}}( \mu_{C} \mid \Q_{C}) = h(T_{C}^{2}, \Part) \leq \sum_{n_1 , n_2} \sum_{a, b, c} -p(n_1, n_2, a, b, c) \log p(n_1, n_2, a, b, c)
\]

We want this estimate to be finite to reasonably combine it with our integral bound.  We will use the usual Jensen's inequality for $\log$ to accomplish this.  Recall that $\sum_{a, b, c} p(n_1, n_2, a, b, c) = p(n_1, n_2)$.  Let $p(n) = \sum_{n_2 = n} p(n_1, n_2)$ and $d = \vert \A(X) \vert^3$.  First we apply the inequality to the inner sum.
\begin{align*}
\sum_{a, b, c} p(n_1, n_2, a, b, c)  \log&\left( \frac{1}{ p(n_1, n_2, a, b, c)} \right) \leq p(n_1 , n_2) \log\left( \frac{d}{p(n_1 , n_2)} \right) \\
&= p(n_1 , n_2) \log(d) + p(n_1 , n_2) \log\left( \frac{1}{p(n_1 , n_2)} \right)
\end{align*}
Now we break the sum over $n_1$ and $n_2$ into a sum first with a fixed value for $n_2$, then over possible choices of $n_2$.  We apply Jensen's inequality once more to the second term above.
\[
\sum_{n_2 = n} p(n_1 , n_2) \log\left( \frac{1}{p(n_1 , n_2)} \right) \leq p(n) \log\left( \frac{n-1}{p(n)} \right)
\]
Combining these inequalities allows us to say
\begin{align*}
 \sum_{n_1 , n_2} \sum_{a, b, c} -p(n_1, n_2, a, b, c) & \log p(n_1, n_2, a, b, c) \leq \\
&= \sum_{n} \left( -p(n) \log p(n) + p(n) \log(n-1) + p(n) \log d \right)
\end{align*}

We will deal with the finiteness of the three terms in this sum separately.  The sum of the third terms is clearly finite.  From the Kac formula (see for example \cite[Theorem 4.6]{Pbook}) , we know the first return times have finite expectation, and thus so do the second return times.  This allows us to say $\sum_{n} n p(n) < \infty$, which implies the second set of terms $\sum_{n} p(n) \log(n-1) < \infty$.  The first term can be dealt with by first defining $A = \left\{ n \mid p(n) \leq \frac{1}{n^2} \right\}$.  Then we have
\[
\sum_{n \in A} -p(n) \log p(n) < \sum_{n \in A} \frac{2 \log n}{n^2} < \infty
\]
When $n \in A^{\mathsf{C}}$, 
\begin{align*}
\log p(n) &> -2 \log n \\
 -p(n) \log p(n) &< 2 p(n) \log n \\
\sum_{n \in A^{\mathsf{C}}} -p(n) \log p(n) < \sum_{n \in A^{\mathsf{C}}} 2 p(n) \log n
\end{align*}
Again from the Kac formula we can say this last sum is finite, and thus the estimate on the relative entropy is finite.

Combining the two bounds we arrive at an inequality of the form
\begin{align*}
\frac{2}{\bar{\mu}(C)} & \left\{ h(\bar{\mu} \mid \bar{\mu} \circ (\pi \circ \pi_{X})^{-1}) + \int \bar{f} \, d\bar{\mu} \right\} \leq \\ 
& \sum_{n_1, n_2} \sum_{A} B(n_1, n_2) p(n_1, n_2, A) - p(n_1, n_2, A) \log p(n_1, n_2, A)
\end{align*}

If we define $a(n_1, n_2) = e ^{B(n_1, n_2)} = \left( \frac{1}{n_2} \right)^{t}$ then for a suitable value of $t$ these terms are summable, and in fact their sum is as small as we like.  Assume that $t$ is chosen to make these summable.  For simplicity in the following argument, we will temporarily reindex $(n_1, n_2, A)$ to $n$. Define $C = \sum_{n} a_{n}$. We are seeking to show $\sum_{n} p_{n} \log \frac{a_n}{p_n}$ is maximized when $\frac{a_n}{p_n} = C$ for all $n$.  The tangent line to $\log$ at $C$ is given by $L_{C}(x) = \log(C) - \frac{1}{C}(x - C)$.  Note that because $\log(x)$ is concave, the tangent at $C$ lies entirely above $\log(x)$.  So, taking $x = \frac{a_{n}}{p_{n}}$ gives us the inequality $\frac{a_n}{p_{n}C} - 1 + \log(C) \geq \log{ \frac{a_n}{p_n}}$.  Thus
\begin{align*}
\sum_{n} p_{n} \log{\frac{a_n}{p_n}} &\leq \sum_{n} p_{n} \left( \frac{a_n}{p_{n}C} - 1 + \log(C) \right) \\
& = \log(C)
\end{align*}
When $\frac{a_n}{p_n} = C$ this bound is obtained.  Indexing back to $(n_1, n_2, A)$ we obtain the following inequality.
\[
\sum_{n_1, n_2} \sum_{A} B(n_1, n_2) p(n_1, n_2, A) - p(n_1, n_2, A) \log p(n_1, n_2, A) \leq \log\left( \sum_{n_1, n_2} \left( \frac{1}{n_2} \right)^{t} \right)
\]
Through adjusting $t$ we can ensure that this bound is negative, and thus the theorem is proved. \qed

\section{Phase transition and open questions}
The following interpretation of Theorems \ref{thm:1dini} and \ref{thm:pdini} seems relevant.  When the only functions we are observing are $1$-Dini, we see that all of their relative equilibrium states are fully supported.  As soon as we relax our view to functions which are $(1+\epsilon)$-Dini, we begin to see some which have non-fully supported relative equilibrium states, including the compensation function constructed in Theorem \ref{thm:pdini}.  This kind of sharp boundary which characterizes a fundamental shift in the behaviour of equilibrium states can be reasonably called a phase transition.  

This work fits into a body of results on infinite-to-one factor maps of shifts of finite type which has seen considerable growth in recent years (for example \cite{PQS}, \cite{AQ}, \cite{SS} and \cite{Yoo}).  These works have developed a greatly improved understanding of what such factor systems look like, but that picture relies on the strong combinatorial properties present when working with a shift of finite type.  One feature which disappears when we consider factors of more general subshifts is the MPW construction.  The existence of Walters-type compensation functions is closely tied to this construction, which motivates the following conjecture.

\begin{conj}
There exist subshifts $X$ and $Y$, and a factor map $\pi:X \to Y$, such that $(X, Y, \pi)$ has no continuous compensation function.
\end{conj}

\bibliography{antonioliCFn}
\bibliographystyle{abbrv}

\end{document}